\documentclass[12pt]{amsart}
\usepackage[
    colorlinks=false,          
    linkbordercolor={1 0 0},   
    citebordercolor={0 1 0},   
    urlbordercolor={0 0 1}     
]{hyperref}

\usepackage{mathtools}
\usepackage{xcolor}
\usepackage{mdwlist}
\usepackage{float}
\allowdisplaybreaks
\usepackage{longtable}
\usepackage{amsfonts,amssymb,amsmath,textcomp}
\usepackage{enumitem}

\usepackage{tikz}
\usepackage{enumitem}

\usepackage[justification=centering]{caption}
\usepackage{supertabular}
 2

\theoremstyle{plain}
\newtheorem{thm}{Theorem}[section]

\newtheorem{prop}[thm]{Proposition}
\newtheorem{cor}[thm]{Corollary}
\newcommand{\thmref}[1]{Theorem~\ref{#1}}

\newcommand{\propref}[1]{Proposition~\ref{#1}}

\theoremstyle{definition}

\newcommand{\rmkref}[1]{Remark~\ref{#1}}
\newtheorem{rmk}{Remark}

\numberwithin{equation}{section}
\begin{document}
\baselineskip 5.1mm
\title[Discriminants and Large P\'olya Groups in Septic Number Fields]
{Discriminants and Large P\'olya Groups in Septic Number Fields}
\author{Nimish Kumar Mahapatra}
\address[Nimish Kumar Mahapatra]{Indian Institute of Science Education and Research Thiruvananthapuram, India.}
\email{nimish@iisertvm.ac.in}


\date{\today}


\begin{abstract}
We investigate a new family of cyclic septic fields $\{K_t\}_{t\in\mathbb{Z}}$ arising from the Hashimoto--Hoshi construction. For this family, we compute the discriminant explicitly and characterize their P\'olya property under the condition that the polynomial $E(t) = t^{6} + 2t^{5} + 11t^{4} + t^{3} + 16t^{2} + 4t + 8$ takes fifth-power free values. We show that this family contains infinitely many non-P\'olya fields for which the cardinality of the P\'olya group is unbounded. We also establish that, assuming Bunyakovsky's conjecture for $E(t)$, this family contains infinitely many P\'olya fields. We further show that, for any fixed positive integer $m$, there exist infinitely many blocks of $m$ consecutive fields in this family whose cardinality of the P\'olya groups can be made arbitrarily large. Finally, we demonstrate that infinitely many fields in this family are non-monogenic with field index one.
\end{abstract}

\maketitle{}

\section{Introduction}

Let $K$ be a number field with ring of integers $\mathcal{O}_K$. The ring of integer-valued polynomials over $\mathcal{O}_K$ is defined as
$$\text{Int}(\mathcal{O}_K) = \{f \in K[X] \mid f(\mathcal{O}_K) \subseteq \mathcal{O}_K\}.$$
A number field $K$ is called a P\'olya field if the $\mathcal{O}_K$-module $\text{Int}(\mathcal{O}_K)$ admits a regular basis, that is, a basis $(f_n)_{n \geq 0}$ where $\deg(f_n) = n$ for all $n \in \mathbb{N} \cup \{0\}$ (see \cite{ZAN}).

For each positive integer $n$, let $\mathfrak{J}_n(K)$ denote the fractional ideal of $\mathcal{O}_K$ generated by the leading coefficients of all polynomials in $\text{Int}(\mathcal{O}_K)$ of degree $n$, together with $0$. A fundamental result establishes that a number field $K$ is a P\'olya field if and only if $\mathfrak{J}_n(K)$ is principal for all $n \geq 1$ (see  \cite{PJC}). This characterization immediately shows that every number field with a trivial class group is a P\'olya field. However, the converse does not hold in general, since every cyclotomic field is a P\'olya field (see  \cite{ZAN}), yet many have non-trivial class groups.

The obstruction to a number field $K$ being P\'olya is measured by its P\'olya group. Let $\text{Cl}(K)$ denote the class group of $K$. For each $n \geq 1$, write 
$[\mathfrak{J}_n(K)]$ for the class of $\mathfrak{J}_n(K)$ in $\text{Cl}(K)$. 
The P\'olya group $\text{Po}(K)$ is then defined as the subgroup of $\text{Cl}(K)$ generated by the classes $[\mathfrak{J}_n(K)]$ for all $n \geq 1$. Thus, $K$ is a P\'olya field if and only if $\text{Po}(K) = \{1\}$.

The systematic classification of P\'olya fields is a central problem in algebraic number theory. Significant progress has been made for low-degree extensions: Zantema \cite{ZAN} completely characterized quadratic P\'olya fields, while Leriche \cite{AL2} provided a complete classification of Galois cubic P\'olya fields. Leriche's work also encompasses cyclic quartic and sextic fields \cite[Theorems 5.1, 6.2]{AL2}. In contrast, the study of how large P\'olya groups can become across different families of number fields has received relatively little attention in the mathematical literature. Only a handful of works have addressed questions about the largeness of P\'olya groups, including studies in biquadratic fields \cite{JAI} and in simplest cubic fields \cite{JAI2}; see also \cite{JAI3}. In recent work, the author and Pandey \cite{NKM} have investigated Lehmer quintic fields with large P\'olya groups. Lehmer quintic fields arise from Lehmer quintic polynomials, which are a ``simple'' parametric family of polynomials discovered by Emma Lehmer \cite{LEH} via explicit transformations of Gaussian periods. Specifically, ``simple" families are characterized by monic polynomials with integer coefficients and constant term 1, whose splitting fields are cyclic over $\mathbb{Q}$. Similar polynomial constructions have also been developed for cubic, quartic, and sextic cases (see \cite{LEH}). Interestingly, no such “simple” families are currently known in degree $7$ or higher. Moreover, to the best of our knowledge, no infinite family of cyclic septic fields admitting a complete and explicit determination of their Pólya groups has previously appeared in the literature. One of the main goals of the present work is to fill this gap.

Hashimoto and Hoshi \cite{HH} adapted Lehmer's ideas in a more general framework to construct a new parametric family of cyclic septic polynomials over $\mathbb{Q}$. The Hashimoto--Hoshi family of cyclic septic polynomials with parameter $t\in \mathbb{Z}$ is defined as
\begin{equation}
f_t(X) = X^7+a_6(t)X^6+a_5(t)X^5+a_4(t)X^4+a_3(t)X^3+a_2(t)X^2+a_1(t)X+a_0(t),    
\end{equation}
where 
\begin{align*}
    a_6(t)= & \: -(t^3 + t^2 + 5 t + 6)\\
    a_5(t)= & \:3(3 t^3 + 3 t^2 + 8 t + 4)\\
    a_4(t)= & \:(t^7 + t^6 + 9 t^5 - 5 t^4 - 15 t^3 - 22 t^2 - 36 t - 8)\\
    a_3(t)= & \:-t (t^7 + 5 t^6 + 12 t^5 + 24 t^4 - 6 t^3 + 2 t^2 - 20 t - 16)\\
    a_2(t)= & \:t^2 (2 t^6 + 7 t^5 + 19 t^4 + 14 t^3 + 2 t^2 + 8 t - 8)\\
    a_1(t)= & \:-t^4 (t^4 + 4 t^3 + 8 t^2 + 4)\\
    a_0(t)= & \:t^7.
\end{align*}
Although this family is not ``simple'' in the sense described above, since its constant term is $t^7$, it nevertheless offers an interesting class for further study.

Let $\theta_t$ be a root of the polynomial $f_t(X)$ and let $K_t = \mathbb{Q}(\theta_t)$ denote the corresponding number field. Since $[K_t:\mathbb{Q}] = 7$, we refer to the collection $\{K_t\}_{t\in\mathbb{Z}}$ as the Hashimoto--Hoshi family of cyclic septic fields. This article provides a complete characterization of the P\'olya property of this family. A key feature of our approach is that the arithmetic of the Pólya groups in the
family we study is governed by the values of a single sextic polynomial
\[
E(t)=t^6+2t^5+11t^4+t^3+16t^2+4t+8,
\]
whose prime divisors simultaneously control the conductor, the discriminant,
and the ambiguous ideal classes. As a result, the problem of constructing cyclic
septic fields with large Pólya groups is reduced to the study of power--free values
and the distribution of prime divisors of $E(t)$. Our main result in this article is the following theorem.
\begin{thm}\label{T1}
Let $\{K_t\}_{t\in\mathbb{Z}}$ be the family of Hashimoto--Hoshi cyclic septic fields and define $E(t) \;=\; t^{6} + 2t^{5} + 11t^{4} + t^{3} + 16t^{2} + 4t + 8$.
Assume that $E(t)$ is fifth-power free, and let $\omega(z)$ denote the number of distinct prime divisors of an integer $z$. Then:
\begin{enumerate}
    \item The P\'olya group of $K_t$ is
    \[
    \text{Po}(K_t) \;\cong\; 
    \begin{cases}
    (\mathbb{Z}/7\mathbb{Z})^{\,\omega(E(t)) - 2}, & \text{ if $t$ is even},\\[6pt]
    (\mathbb{Z}/7\mathbb{Z})^{\,\omega(E(t)) - 1}, & \text{ if $t$ is odd}.
    \end{cases}
    \]
    In particular, $K_t$ is a P\'olya field precisely when 
    \[
    \omega(E(t))=
    \begin{cases}
    2, & \text{if $t$ is even},\\
    1, & \text{if $t$ is odd}.
    \end{cases}
    \]
    \item The family $\{K_t\}_{t\in\mathbb{Z}}$ contains infinitely many non-P\'olya fields 
    with large P\'olya groups. In particular, the $7$--rank of their class groups is unbounded.
\end{enumerate}
\end{thm}
We define a collection of $m$ Hashimoto–Hoshi cyclic septic fields to be consecutive if they are of the form $K_{t+i}$ for $i = 0, 1, \ldots, m-1$ and some integer $t$. In \cite{JAI2}, the authors investigated the largeness of P\'olya groups occurring in certain consecutive biquadratic fields. Motivated by this work, we study the same question for consecutive fields in the Hashimoto--Hoshi collection of cyclic septic fields. Our result establishes the following.
\begin{thm}\label{T3}
Let \(m\ge 2\) and \(r>1\) be fixed integers. Then there exist infinitely many blocks of
\(m\) consecutive integers \(t,t+1,\dots,t+m-1\) such that
\[
\text{Po}(K_{t+i})\supseteq (\mathbb{Z}/7\mathbb{Z})^{\,r}
\qquad\text{for every }i=0,1,\dots,m-1.
\]
Equivalently, all \(m\) consecutive P\'olya groups in the block have $7$--rank at least \(r\).
\end{thm}

The article is organized as follows. Section 2 presents the necessary preliminaries for establishing our main result. Section 3 provides an explicit computation of the discriminant for the Hashimoto--Hoshi family of cyclic septic fields. Section 4 contains all the proofs of our main theorems. Section 5 examines the notion of P\'olya numbers and the monogeneity properties of this family. In Section 6, we present the concluding remarks and discuss their implications for a long-standing conjecture. Finally, the appendix includes the \textsc{Maple} scripts used to verify the algebraic identities and the polynomial computations presented in this paper.

\section{Preliminaries}

Let $K/\mathbb{Q}$ be a cyclic extension of degree $p$, where $p$ is an odd prime, defined by a monic polynomial of degree $p$ in $\mathbb{Z}[X]$. The following theorem of Spearman and Williams \cite{BKS} provides explicit descriptions for the conductor $f(K)$ and discriminant $d(K)$.

\begin{thm}\cite[Theorem 1]{BKS}\label{CT1} 
Let 
$f(X) = X^p + a_{p-2}X^{p-2} + \cdots + a_1X + a_0 \in \mathbb{Z}[X]$
be such that
\begin{enumerate}
\item $\operatorname{Gal}(f) \simeq \mathbb{Z}/p\mathbb{Z}$, and
\item there does not exist a prime $q$ such that 
\[
q^{\,p-i} \mid a_i, \quad i=0,1,\ldots,p-2.
\]
\end{enumerate}
Let $\theta \in \mathbb{C}$ be a root of $f(X)$ and set $K=\mathbb{Q}(\theta)$ so that 
$K$ is a cyclic extension of $\mathbb{Q}$ with $[K:\mathbb{Q}]=p$. Then
\[
d(K) = f(K)^{p-1},
\]
where the conductor $f(K)$ of $K$ is given by
\[
f(K) = p^\alpha \prod_{\substack{q \equiv 1 \pmod{p} \\ q \mid a_i,\; i=0,1,\ldots,p-2}} q,
\]
where $q$ runs through primes, and
\[
\alpha =
\begin{cases}
0, &\text{if \quad} p^{p(p-1)} \nmid \operatorname{disc}(f)
      \text{ and } p \mid a_i,\; i=1,\ldots,p-2 \text{ does not hold, or},\\
   & \qquad p^{p(p-1)} \mid \operatorname{disc}(f)
      \text{ and } p^{\,p-1}\parallel a_0,\; p^{\,p-1} \mid a_1,\; p^{\,p+1-i} \mid a_i,\\
   &\qquad\qquad\qquad\qquad\text{\:\:\:\:and } \; i=2,\ldots,p-2
      \text{ does not hold},\\[6pt]
2, &\text{if \quad} p^{p(p-1)} \nmid \operatorname{disc}(f)
      \text{ and } p \mid a_i,\; i=1,\ldots,p-2 \text{ holds, or},\\
   &\qquad p^{p(p-1)} \mid \operatorname{disc}(f)
      \text{ and } p^{\,p-1}\parallel a_0,\; p^{\,p-1} \mid a_1,\; p^{\,p+1-i} \mid a_i,\\
   &\qquad\qquad\qquad\qquad\text{\:\:\:\:and } \; i=2,\ldots,p-2
      \text{ holds}.
\end{cases}
\]

\end{thm}
Let $K$ be a finite Galois extension of $\mathbb{Q}$ and for any prime number $p$, we denote the ramification index of $p$ in $K/\mathbb{Q}$ by $e_p$. In \cite{JLC}, Chabert obtained a nice description for the cardinality of $\text{Po}(K)$ for cyclic extensions $K/\mathbb{Q}$.
\begin{prop}\label{CHABERT}
\cite[Corollary 3.11]{JLC} Assume that the extension $K/\mathbb{Q}$ is cyclic of degree $n$.
\begin{enumerate}
    \item If $K$ is real and $N(\mathcal{O}_K^{\times})=\{1\}$, then $|\text{Po}(K)|=\frac{1}{2n}\times\prod_p e_p$.
    \item In all other cases, $|\text{Po}(K)|=\frac{1}{n}\times\prod_p e_p$.
\end{enumerate}
\end{prop}
Zantema \cite[\S 3]{ZAN} showed that $\text{Po}(K)$ is the subgroup of $\text{Cl}(K)$ generated by the classes of the ambiguous ideals of $K$.  In other words, 
\begin{equation}\label{amb}
    \text{Po}(K)=\{[\mathfrak{a}]\in \text{Cl}(K) : \mathfrak{a}^\tau=\mathfrak{a}\text{ for all }\tau\in \text{Gal}(K/\mathbb{Q})\}.
\end{equation}
We quote the following result due to Erd\"os \cite{PER} which plays a crucial role in the proof of our main theorem.
\begin{thm}\label{ero}
\cite[\S 1, Theorem ]{PER} Let $f(x)$ be a polynomial of degree $l\geq3$ whose coefficients are integers with the highest common factor 1 and the leading coefficient is positive. Assume that $f(x)$ is not divisible by the $(l-1)$-th power of a linear polynomial with integer coefficients. Then there are infinitely many positive integers $n$ for which $f(n)$ is $(l-1)$-th power free.
\end{thm}
We now state a fundamental result on power-free values of polynomials.
\begin{thm}\label{Hel}\cite[Theorem 1]{TREUSS}
Let $f(x)\in\mathbb{Z}[x]$ be an irreducible polynomial of degree $d \geq3$ and assume that
$f$ has no fixed $(d-1)$th power prime divisor. Define
$$N_f'(X)=\#\{p\leq X : p \text{ prime}, f(p) \text{ is } (d-1)\text{-free}\}.$$
Then, for any $C>1$, we have 
$$N_f'(X)=c_f'\pi(X)+O_{C,f}\left(\frac{X}{(\log X)^C}\right),$$
as $x\rightarrow\infty$, where 
$$c_f'=\prod_{p}\left(1-\frac{\rho'(p^{d-1})}{\phi(p^{d-1})}\right)$$
and 
$$\rho'(d)=\#\{n\!\!\!\!\!\pmod d : (d,n)=1, d \mid f(n) \}.$$
\end{thm}
Let $f(x)$ be an irreducible polynomial with integral coefficients and $f(m)>0$ for $m=1,2,\dots$. Let $\omega(m)$ denote the number of distinct primes dividing $m$. Then, for primes $p$, the following result due to Halberstam \cite{HAL} determines the distribution of values of $\omega(f(p))$.
\begin{thm}\label{ham} \cite[Theorem 2]{HAL}
Let $f(X) \in \mathbb{Z}[X]$ be any non-constant polynomial. For all but $o(X/ \log X)$ primes $p \leq X$,
\begin{equation*}
    \omega\Big(f(p)\Big)=\Big(1+o(1)\Big){\log \log X}.
\end{equation*}
\end{thm}
The following result of Gras \cite{MNG} addresses the monogenicity of cyclic number fields of prime degree.
\begin{prop}\label{GRAS}
\cite[\S 5, Théorème]{MNG} If $K$ is a cyclic number field of prime degree $p\geq5$, then $K$ is monogenic only if $2p+1$ is prime and it is the maximal real subfield of the cyclotomic field $\mathbb{Q}(\zeta_{2p+1})$.
\end{prop}

\section{Discriminant of Hashimoto--Hoshi Cyclic Septic Fields}

In this section, we explicitly evaluate the discriminant of the Hashimoto--Hoshi family of cyclic septic fields using results of Spearman and Williams \cite{BKS}. We begin by introducing the following notation.
\begin{align}
    E &= E(t) = t^{6} + 2t^{5} + 11t^{4} + t^{3} + 16t^{2} + 4t + 8 \notag\\
    F &= F(t) = 10t^{3} + 10t^{2} + t + 4 \notag\\
    G &= G(t) = 15t^{6} + 30t^{5} - 31t^{4} + 15t^{3} - 201t^{2} - 87t - 174 \notag\\
    H &= H(t) = 6t^{15} + 30t^{14} - 133t^{13} - 504t^{12} - 3255t^{11} - 6244t^{10} \notag\\
    &\quad - 12033t^{9} - 8438t^{8} + 19620t^{7} + 52892t^{6} + 136787t^{5} \notag\\
    &\quad + 167671t^{4} + 179676t^{3} + 206640t^{2} + 103680t + 82944\label{polynomials}\\
    I &= I(t) = 12t^{9} + 36t^{8} - 78t^{7} - 84t^{6} - 861t^{5} - 588t^{4} \notag\\
    &\quad - 1155t^{3} - 1214t^{2} - 324t - 432 \notag\\
    J &= J(t) = 5t^{12} + 20t^{11} - 66t^{10} - 162t^{9} - 1126t^{8} - 1441t^{7} \notag\\
    &\quad - 2534t^{6} - 1641t^{5} + 1857t^{4} + 426t^{3} + 5574t^{2} + 3456t + 3456. \notag
\end{align}
As an application of \thmref{CT1}, we establish the following result.

\begin{thm}\label{T2}
With the notation above, let $K_t$ be a Hashimoto--Hoshi cyclic septic field. 
Then the discriminant $d(K_t)$ is given by
\begin{equation*}
    d(K_t)=(f(K_t))^6,
\end{equation*}
where the conductor $f(K_t)$ is given by 
\begin{equation*}
    f(K_t) = 7^{\alpha}\prod\limits_{\substack{q \equiv 1 \pmod{7} \\ q \mid E \\ v_q(E) \not\equiv 0 \pmod{7}}}q,
\end{equation*}
where the product runs over all primes $q$, and 
\begin{equation*}
\alpha = \begin{cases}
0 & \text{if } t \not\equiv 2 \pmod{7},\\
2 & \text{if } t \equiv 2 \pmod{7}.
\end{cases}
\end{equation*}
\end{thm}
\begin{proof}
Applying a Tschirnhausen transformation to eliminate the $X^6$--term of $f_t(X)$, we obtain the following.  
\begin{align*}
f^*_t(X) =& \: 7^7f_t((X+(t^3 + t^2 + 5 t + 6))/7)\\
=& \: X^7+h_5(t)X^5+h_4(t)X^4+h_3(t)X^3+h_2(t)X^2+h_1(t)X+h_0(t),     
\end{align*}
where 
\begin{align*}
    h_5 = & -21t^{6} - 42t^{5} - 231t^{4} - 21t^{3} - 336t^{2} - 84t - 168\\
    h_4 = & -70t^{9} - 210t^{8} - 917t^{7} - 882t^{6} - 1323t^{5} - 1715t^{4} - 980t^{3} \\
    & - 1036t^{2} - 168t - 224\\
    h_3 = & -105t^{12} - 420t^{11} - 1358t^{10} - 2086t^{9} + 1694t^{8} - 1295t^{7} + 19600t^{6} \\
    & + 8050t^{5} + 37835t^{4} + 15750t^{3} + 33180t^{2} + 9744t + 9744\\
    h_2 = & -84t^{15} - 420t^{14} - 882t^{13} - 1176t^{12} + 11613t^{11} + 18816t^{10} \\
    & + 90258t^{9} + 85547t^{8} + 215467t^{7} + 203791t^{6} + 233534t^{5} \\
    & + 236768t^{4} + 137984t^{3} + 125440t^{2} + 30240t + 24192\\
    h_1 = & -35t^{18} - 210t^{17} - 203t^{16} + 483t^{15} + 14532t^{14} + 36407t^{13} \\
    & + 132300t^{12} + 184674t^{11} + 349524t^{10} + 317107t^{9} + 210728t^{8} \\
    & + 187411t^{7} - 524888t^{6} - 361326t^{5} - 1030512t^{4} - 591192t^{3} \\
    & - 795984t^{2} - 290304t - 193536\\
    h_0 = & -6t^{21} - 42t^{20} + 7t^{19} + 434t^{18} + 5600t^{17} + 17927t^{16} + 62790t^{15} \\
    & + 112799t^{14} + 191023t^{13} + 129675t^{12} - 206409t^{11} - 809585t^{10} \\
    & - 2256471t^{9} - 3404408t^{8} - 5218187t^{7} - 6396040t^{6} - 6152804t^{5} \\
    & - 6382376t^{4} - 4005792t^{3} - 3394944t^{2} - 1161216t - 663552
\end{align*}
From (\ref{polynomials}) we obtain  
\begin{equation}\label{common}
\begin{aligned}
h_5 &= -21E, & \quad h_4 &= -7FE, & \quad h_3 &= -7GE, \\[4pt]
h_2 &= -7IE, & \quad h_1 &= -7JE, & \quad h_0 &= -HE.
\end{aligned}
\end{equation}
Next, let $m$ be the largest positive integer such that  
\begin{equation}\label{dividem}
    m^2 \mid h_5, \qquad m^3 \mid h_4, \qquad m^4 \mid h_3, \qquad 
    m^5 \mid h_2, \qquad m^6 \mid h_1, \qquad m^7 \mid h_0.
\end{equation}
We then define  
\begin{equation}\label{kt-def}
    g_t(X) \;=\; \frac{f^*_t(mX)}{m^5}
    \;=\; X^7 + k_5X^5 + k_4X^4 + k_3X^3 + k_2X^2 + k_1X + k_0,
\end{equation}
where  
\begin{equation}\label{reduce}
    k_5 = \tfrac{h_5}{m^2},  \quad k_4 = \tfrac{h_4}{m^3},  \quad k_3 = \tfrac{h_3}{m^4}, 
    \quad k_2 = \tfrac{h_2}{m^5},  \quad k_1 = \tfrac{h_1}{m^6},  \quad k_0 = \tfrac{h_0}{m^7}.
\end{equation}
Using \textsc{Maple}, we obtain  
\begin{equation}\label{disc}
\operatorname{disc}(g_t) \;=\; 
7^{42} E^{6} t^{22}
\bigl(2t^{4}-2t^{3}+6t^{2}-3t+4\bigr)^{2}
\bigl(t^{5}+t^{4}+t^{3}+2t^{2}+t+1\bigr)^{2} \;/\; m^{18}.
\end{equation}
Clearly, $g_t(X)$ is a defining polynomial for the cyclic septic field $K_t$.  
Therefore, using \thmref{CT1}, we obtain
\begin{equation}\label{maincond}
     f(K)=7^\alpha \prod\limits_{\substack{q\equiv 1\pmod 7 \\ q\mid k_0, \: q\mid k_1, \: q\mid k_2,\:  q\mid k_3, \: q\mid k_4, \: q\mid k_5 }}q
 \end{equation}
 where $q$ runs over primes, and 
 \begin{equation}\label{cond}
\alpha= \begin{cases}
     0, \quad\text{ if }7^{42}\nmid \text{ disc } (g_t) \text{ and } 7\mid k_1, 7\mid k_2, 7\mid k_3, 7\mid k_4, 7\mid k_5\\
     \qquad\text{does not hold, or}\\
     \qquad7^{42}\mid \text{ disc } (g_t) \text{ and } 7^6||k_0, 7^6\mid k_1, 7^5\mid k_3, 7^4\mid k_4, 7^3\mid k_5\\
     \qquad\text{does not hold,}\\
     2, \quad\text{ if }7^{42}\nmid \text{ disc } (g_t) \text{ and } 7\mid k_1, 7\mid k_2, 7\mid k_3, 7\mid k_4, 7\mid k_5\\
     \qquad\text{holds, or}\\
     \qquad7^{42}\mid \text{ disc } (g_t) \text{ and } 7^6||k_0, 7^6\mid k_1, 7^5\mid k_3, 7^4\mid k_4, 7^3\mid k_5\\
     \qquad\text{holds.}\\
 \end{cases}     
 \end{equation}
Let $q$ be a prime with 
\begin{equation*}
    q\equiv1\pmod{7}, \quad q\mid k_0, \: q\mid k_1, \: q\mid k_2,\:  q\mid k_3, \: q\mid k_4, \: q\mid k_5. 
\end{equation*}
We show that 
\begin{equation*}
    q\mid E, \quad v_q(E)\not\equiv0 \pmod{7}.
\end{equation*}
By (\ref{reduce}) we have
\begin{equation*}
    q\mid h_5, \: q\mid h_4, \: q\mid h_3, \: q\mid h_2, \: q\mid h_1, \: q\mid h_0.
\end{equation*}
Since $q \equiv 1 \pmod{7}$, it follows that $q \neq 3,7$. Thus, from (\ref{common}), we deduce $q \mid E$. A direct computation in \textsc{Maple} shows that 
\[
\gcd(E,H) = 1.
\]
Hence there exist polynomials $\widetilde{L}$, $\widetilde{M} \in \mathbb{Q}[t]$ with 
\[
E\widetilde{L} + H\widetilde{M} = 1.
\]
In particular, we can choose $L,M \in \mathbb{Z}[t]$ such that 
\begin{equation}\label{eq:bezout}
EL + HM = 26353376 = 2^{5}\cdot 7^{7}.
\end{equation}
An explicit pair is given by
\[
\begin{aligned}
L(t) &= -150t^{14} - 492t^{13} + 5035t^{12} + 7709t^{11} + 43169t^{10} \\
&\quad + 6075t^{9} - 61261t^{8} - 3409t^{7} - 53980t^{6} + 1520256t^{5} \\
&\quad + 1587001t^{4} - 1436918t^{3} - 5167741t^{2} - 16768814t - 14725412, \\[6pt]
M(t) &= 25t^{5} + 7t^{4} + 119t^{3} - 376t^{2} + 155t + 1738.
\end{aligned}
\]
Since $q \mid E$ but $q \nmid 2,7$, it follows from (\ref{eq:bezout}) that $q \nmid H$.  
Suppose, for the sake of contradiction, that $v_q(E) \equiv 0 \pmod{7}$, say $v_q(E)=7w$ with $w \geq 1$.  
Then, by (\ref{common}), we have 
\begin{equation*}
    q^{7w} \parallel h_5, \qquad 
    q^{7w} \mid h_4, \qquad 
    q^{7w} \mid h_3, \qquad
    q^{7w} \mid h_2, \qquad
    q^{7w} \mid h_1, \qquad
    q^{7w} \parallel h_0,
\end{equation*}
and hence, from (\ref{dividem}), it follows that 
\begin{equation*}
    q^w \parallel m.
\end{equation*}
Therefore, by (\ref{reduce}),
\begin{equation*}
    q \nmid \frac{h_0}{m^7} = k_0,
\end{equation*}
which is a contradiction.  
Thus we conclude that $v_q(E) \not\equiv 0 \pmod{7}$.\\
\noindent
Conversely, let $q$ be a prime such that 
\begin{equation*}
    q \equiv 1 \pmod{7}, \qquad q \mid E, \qquad v_q(E) \not\equiv 0 \pmod{7}.
\end{equation*}
We claim that 
\begin{equation*}
    q \mid k_0, \quad q \mid k_1, \quad q \mid k_2, \quad q \mid k_3, \quad q \mid k_4, \quad q \mid k_5.
\end{equation*}
From (\ref{eq:bezout}), we know that $q \nmid H$.  
Since $v_q(E) \not\equiv 0 \pmod{7}$, we may write 
\[
v_q(E) = 7z + r, \qquad z \geq 0, \quad r \in \{1,2,3,4,5,6\}.
\]
Then, by (\ref{common}), it follows that
\begin{equation*}
    q^{7z+r} \parallel h_5, \quad 
    q^{7z+r} \mid h_4, \quad 
    q^{7z+r} \mid h_3, \quad 
    q^{7z+r} \mid h_2, \quad 
    q^{7z+r} \mid h_1, \quad 
    q^{7z+r} \parallel h_0.
\end{equation*}
Hence, from (\ref{dividem}), we obtain
\begin{equation*}
    q^z \parallel m,
\end{equation*}
and therefore, by (\ref{reduce}),
\begin{equation*}
    q^{5z+r} \parallel k_5, \quad 
    q^{4z+r} \mid k_4, \quad 
    q^{3z+r} \mid k_3, \quad 
    q^{2z+r} \mid k_2, \quad 
    q^{z+r} \mid k_1, \quad 
    q^r \parallel k_0.
\end{equation*}
In particular,
\begin{equation*}
    q \mid k_0, \quad q \mid k_1, \quad q \mid k_2, \quad q \mid k_3, \quad q \mid k_4, \quad q \mid k_5,
\end{equation*}
as required. We have established that 
\begin{equation}\label{eq:product}
\prod_{\substack{q \equiv 1 \pmod{7} \\ q \mid k_0, \, q \mid k_1, \, q \mid k_2, \, q \mid k_3, \, q \mid k_4, \, q \mid k_5}}
    q \;=\;
\prod_{\substack{q \equiv 1 \pmod{7} \\ q \mid E \\ v_q(E) \not\equiv 0 \pmod{7}}} q.
\end{equation}
Finally, to complete the proof, we claim that the exponent $\alpha$ takes the following values in the respective cases:
\begin{equation}\label{eq:alpha}
    \alpha = 
    \begin{cases}
        0, & \text{if } t \not\equiv 2 \pmod{7}, \\[6pt]
        2, & \text{if } t \equiv 2 \pmod{7}.
    \end{cases}
\end{equation}
Let $t \not\equiv 2 \pmod{7}$. In this case, it is straightforward to verify that 
\[
v_7(E) = v_7(F) = v_7(G) = v_7(H)= v_7(I) = v_7(J)  = 0.
\]
From (\ref{common}), it follows that
\[
7 \parallel h_5, \quad 7 \parallel h_4, \quad 7 \parallel h_3, \quad 7 \parallel h_2, \quad 7 \parallel h_1, \quad 7 \nmid h_0.
\]
Consequently, we deduce that $7 \nmid m$, and therefore
\[
7 \parallel k_5, \quad 7 \parallel k_4, \quad 7 \parallel k_3, \quad 7 \parallel k_2, \quad 7 \parallel k_1, \quad 7 \nmid k_0.
\]
Hence, by (\ref{cond}), we conclude that $\alpha = 0$.\\
\noindent
Let $t \equiv 2 \pmod{7}$. In this case, using \textsc{Maple} we obtain that 
\begin{equation}\label{value}
\begin{aligned}
&v_7(E) = 2, \quad &v_7(F) = 1, \quad &v_7(G) = 2, \\
\quad &v_7(I) = 3, \quad &v_7(J) = 4, \quad &v_7(H) = 5.    
\end{aligned}    
\end{equation}
Hence, by (\ref{common}),
\[
7^3 \parallel h_5, \quad 7^4 \parallel h_4, \quad 7^5 \parallel h_3, \quad 7^6 \parallel h_2, \quad 7^7 \parallel h_1, \quad 7^7 \parallel h_0.
\]
From (\ref{dividem}) it follows that 
\[
7 \parallel m.
\]
Therefore, by (\ref{reduce}),
\[
7 \parallel k_5, \quad 7 \parallel k_4, \quad 7 \parallel k_3, \quad 7 \parallel k_2, \quad 7 \parallel k_1, \quad 7 \nmid k_0.
\]
Moreover, by (\ref{disc}) and a direct computation in \textsc{Maple}, we obtain 
\[
7^{36} \parallel \mathrm{disc}(g_t).
\]
Thus, by (\ref{cond}), we conclude that $\alpha = 2$. The proof is now complete by combining (\ref{maincond}), (\ref{cond}), (\ref{eq:product}), and (\ref{eq:alpha}).

\end{proof}

\section{Proofs}

This section provides a complete proof of our main theorems along with some relevant remarks. The following proposition is crucial for the proof of Theorem \ref{T1}.
\begin{prop}\label{2prop1}
    Let $t$ be any integer and $p>7$ be any prime, dividing $E(t)$. Then the congruence $p\equiv 1 \pmod7$ holds.
\end{prop}

\begin{proof}
Let $\Phi_7(x)=x^6+x^5+x^4+x^3+x^2+x+1$ denote the $7$th cyclotomic polynomial, and define
\[
w(x) = \frac{x^5}{7} + \frac{3x^4}{14} + \frac{9x^3}{7} - \frac{11x^2}{14} + \frac{13x}{14}.
\]
A direct computation shows that there exists a polynomial $W(x) \in \mathbb Z[x]$ such that
\begin{equation}\label{neww}
\Phi_7(w(x)) = \frac{E(x)\, W(x)}{2^6 7^6},    
\end{equation}
where
\begin{align*}
W(x) &= 64 x^{24} + 448 x^{23} + 4016 x^{22} + 15104 x^{21} + 71180 x^{20} + 142780 x^{19} \\
&\quad + 480245 x^{18} + 258262 x^{17} + 1611259 x^{16} - 1664243 x^{15} + 6281164 x^{14} \\
&\quad - 9116196 x^{13} + 20556660 x^{12} - 28970654 x^{11} + 33948931 x^{10} - 21216648 x^9 \\
&\quad + 5083933 x^8 + 7056657 x^7 - 3702019 x^6 - 2324868 x^5 + 4160443 x^4 \\
&\quad + 671937 x^3 - 2012038 x^2 + 403368 x + 941192.
\end{align*}
Now, let $t\in \mathbb Z$ and suppose $p>7$ is a prime with $p \mid E(t)$.  
Since $p \not\in \{2,7\}$, the number $14$ is invertible modulo $p$, so $w(t)$ is well-defined modulo $p$.  
Reducing the above identity modulo $p$ and using $E(t)\equiv 0 \pmod p$, we obtain
\[
\Phi_7(w(t)) \equiv 0 \pmod p.
\]
Let $t_0 \in \mathbb Z$ be any integer with $t_0 \equiv w(t) \pmod p$. Then $\Phi_7(t_0) \equiv 0 \pmod p$.  
We claim that $t_0 \not\equiv 1 \pmod p$. Suppose, for contradiction, that $t_0 \equiv 1 \pmod p$. Then
\[
\Phi_7(t_0) \equiv \Phi_7(1) = 7 \not\equiv 0 \pmod p,
\]
which contradicts $p \mid \Phi_7(t_0)$. Hence $t_0 \not\equiv 1 \pmod p$. Finally, since $\Phi_7(t_0) \equiv 0 \pmod p$ and $t_0 \not\equiv 1 \pmod p$, we have
\[
t_0^7 \equiv 1 \pmod p
\]
and the multiplicative order of $t_0$ modulo $p$ is exactly $7$.  
Therefore, $7 \mid (p-1)$, which implies
\(
p \equiv 1 \pmod 7.
\)
This completes the proof.
\end{proof}

\begin{proof}[Proof of $\thmref{T1}$]
We consider the set $$\mathcal{S}=\{t \in \mathbb{Z}: E(t)\mbox{ is a fifth-power free integer} \}.$$  
For $t \in \mathcal{S}$, we have
 \begin{equation}\label{2e1}
   E(t)=7^{\alpha}AB^2C^3D^4.
 \end{equation}
 Here $\alpha=2$ if $t\equiv2 \pmod{7}$ and $\alpha=0$ otherwise, and $A, B, C$ and $D$ are natural numbers that are relatively prime and $7 \nmid ABCD$.  Note that for any integer $t$ we have $3\nmid E(t)$ and $5\nmid E(t)$. If $t$ is even, then $2^3||E(t)$ but $2\nmid f(K_t)$ and hence $2\nmid d(K_t)$.  Using \thmref{T2} and \propref{2prop1} we have 
\begin{align}
     f(K_t) &= \begin{cases}
     7^{\alpha}AB(C/2)D, & \text{if $t$ is even},\\
     7^{\alpha}ABCD,  & \text{if $t$ is odd},
     \end{cases}\\
     \text{and } \quad d(K_t) &= \begin{cases}
     \left(7^{\alpha}AB(C/2)D\right)^6, & \text{if $t$ is even},\\
     (7^{\alpha}ABCD)^6,  & \text{if $t$ is odd}.
     \end{cases} 
\end{align}
 Since $K_t / \mathbb{Q}$ is Galois and of degree $7$, for any prime $p$ the ramification index $e_p$ of $p$ in $K_t$ is given by
 \begin{align*}e_p=
 \begin{cases}
7, & \text{if } p\mid d(K_t),\\
1, & \text{otherwise}.
 \end{cases}
 \end{align*}
Thus, we obtain
\begin{equation}
    \prod_{p} e_p \;=\; 7^{\omega(d(K_t))} \;=\;
    \begin{cases}
    7^{\,\omega(E(t)) - 1}, & \text{if $t$ is even}, \\[6pt]
    7^{\,\omega(E(t))},     & \text{if $t$ is odd},
    \end{cases}
\end{equation}
Now from \propref{CHABERT} we get 
\begin{equation}\label{2pol}
    |\text{Po}(K_t)|=\begin{cases}
     7^{\omega(E(t))-2},\quad \;\text{ if $t$ is even}, \\
     7^{\omega(E(t))-1}, \quad\text{\; if $t$ is odd}. 
    \end{cases}
\end{equation}
We have $\text{Gal}(K_t/\mathbb{Q})\simeq \mathbb{Z}/7\mathbb{Z}$. Let $\sigma$ be a generator of $\text{Gal}(K_t/\mathbb{Q})$ and $[\mathfrak{I}]\neq [1]$ be an ambiguous ideal class in $K_t$. Then we have,
\begin{align*}
    [\mathfrak{I}]^7&=[\mathfrak{I}][\mathfrak{I}][\mathfrak{I}][\mathfrak{I}][\mathfrak{I}][\mathfrak{I}][\mathfrak{I}]\\    
    &=[\mathfrak{I}][\mathfrak{I}]^\sigma[\mathfrak{I}]^{\sigma^2}[\mathfrak{I}]^{\sigma^3}[\mathfrak{I}]^{\sigma^4}[\mathfrak{I}]^{\sigma^5}[\mathfrak{I}]^{\sigma^6}\\
    &=[\delta\mathcal{O}_{K_t}]\\
    &=[(1)]
\end{align*}
where $|\delta|=N(\mathfrak{I})=[\mathcal{O}_{K_t}: \mathfrak{I}]$ denotes the norm of the ideal $\mathfrak{I}$ and $[(1)]$ denotes the trivial ideal class of $\text{Cl}_{K_t}$. We conclude that the order of any non-trivial ambiguous ideal class in the class group of $K_t$ is $7$. From \eqref{amb} and the structure theorem for the abelian groups, we obtain $$\text{Po}(K_t)\cong\begin{cases}
\left(\frac{\mathbb{Z}}{7\mathbb{Z}}\right)^{\omega({E(t)})-2} & \text{if $t$ is even},\\
\left(\frac{\mathbb{Z}}{7\mathbb{Z}}\right)^{\omega({E(t)})-1} & \text{if $t$ is odd}.
\end{cases}$$
This finishes the proof of \thmref{T1}(1).

Let $t\equiv2\pmod{7}$ be an integer. Then, from equation (\ref{2e1}), we see that $E(t)$ can neither be a prime nor a prime power. Thus, for all such $t$, $K_t$ is a non-P\'olya field whenever $t \in \mathcal{S}$.
Now, we show that there are infinitely many such $t$'s. To do this, we show that there are infinitely many integers $s$ such that $E(7s+2)$ is fifth-power free. Expanding we obtain 
\begin{equation}
      E(7s+2)=7^2(2401s^6+4802s^5+4459s^4+2303s^3+686s^2+112s+8)=7^2\psi(s),
\end{equation}
where $\psi(x)=2401x^6+4802x^5+4459x^4+2303x^3+686x^2+112x+8$ and $7 \nmid \psi(s)$ for any integer $s$.
If $(ax+b)^5\mid\psi(x)$ for some $a,b\in\mathbb Z$ then $u=-b/a\in\mathbb Q$ is a root of multiplicity at least $5$, so
\[
\psi(u)=\psi'(u)=\psi''(u)=\psi^{(3)}(u)=\psi^{(4)}(u)=0.
\]
Equivalently, $(x-u)$ divides $\gcd(\psi,\psi^{(k)})$ for $k=1,2,3,4$. 
However, a direct computation shows
\[
\gcd(\psi,\psi')=\gcd(\psi,\psi'')=\gcd(\psi,\psi^{(3)})=\gcd(\psi,\psi^{(4)})=1
\]
in $\mathbb Z[x]$, so $\psi$ has no repeated roots. Hence no such linear factor $(ax+b)$ can occur such that $(ax+b)^5\mid\psi(x)$ for any integers $a,b$. Thus, from Theorem \ref{ero}, it follows that $\psi(s)$ is fifth-power free for infinitely many $s$. Since $7\nmid \psi(s)$ for any integer $s$, it follows that $E(7s+2)$ is fifth-power free for infinitely many integers $s$. To show the largeness of the size of the P\'olya groups, it is enough to find integers $t$ such that $E(7t+2)$ is fifth-power free and $\omega(E(t))$ goes to infinity as $t$ goes to infinity. Clearly $\psi(x)$ is irreducible in $\mathbb{Q}[x]$ and from Theorem \ref{Hel} for a positive proportion of primes $p$ we have $\psi(p)$ is fifth-power free. Consequently, $E(7p+2)$ is fifth-power free for a positive proportion of prime numbers $p$. Put
 $$\mathcal{S'}=\{p: \psi(p) \mbox{ is a fifth-power free integer} \}.$$
Then for $p\in\mathcal{S'}$, $\text{Po}(K_{7p+2})\cong (\mathbb{Z}/7\mathbb{Z})^{\omega(E({7p+2}))-1}.$ Since $\mathcal{S'}$ has positive proportion among primes, from \thmref{ham} it follows that the size of $\text{Po}(K_{7p+2})$ is unbounded and hence the 7-rank of their class groups $\text{Cl}(K_{7p+2})$ is unbounded. This completes the proof of \thmref{T1}(2).
\end{proof}

\begin{rmk}
From the proof of Theorem~\ref{T1}, it follows that if $E(7t+2)$ is fifth–power free, then the field $K_{7t+2}$ is non-P\'olya. The converse is not true in general, as there exist non-P\'olya fields with $E(7t+2)$ divisible by a fifth power. For instance, when $t=21342$, we have $7t+2 = 149396$ and $29^5$ divides $E(7\cdot 21342+2)$.
\end{rmk}

\begin{rmk}
It is conjectured \cite[Conjecture~1.1]{JEA} that for integers $d>1$ and $m>1$, the $m$-rank of the class group of a number field $K$ is unbounded as $K$ ranges over fields of degree $d$. This conjecture is known to hold when $m=d$, or more generally when $m\mid d$, by class field theory (see \cite[Conjecture~1.1]{JEA}). Eventually, \thmref{T1}(2) provides an alternative proof of the conjecture in the special case $m=d=7$.
\end{rmk}

\begin{rmk}\label{r3}
Let $t$ be odd (the case of even $t$ is analogous) and suppose that $E(t)$ is fifth–power free. Since $E(t)$ is irreducible over $\mathbb{Q}$, has a positive leading coefficient, and no fixed prime divisor, Bunyakovsky's conjecture predicts that $E(t)$ takes prime values for infinitely many integers $t$. Consequently, one expects that infinitely many members of the family $\{K_t\}_{t\in\mathbb{Z}}$ are P\'olya fields.  
\end{rmk}

\begin{proof}[Proof of $\thmref{T3}$] We have 
\[
E(t)=t^{6}+2t^{5}+11t^{4}+t^{3}+16t^{2}+4t+8 \in \mathbb Z[t].
\]
For even integers \(t=2k\), we define
\[
U(k):=\frac{E(2k)}{2^3}=8k^{6}+8k^{5}+22k^{4}+k^{3}+8k^{2}+k+1\in\mathbb Z[k],
\]
and for odd integers \(t=2k+1\), we define
\[
V(k):=E(2k+1)\in\mathbb Z[k].
\]
A direct reduction modulo \(2\) shows that \(U(k)\) is odd for every \(k\in\mathbb Z\). In particular, for all \(k\)
\begin{equation}\label{eq:omega-even}
\omega\bigl(E(2k)\bigr)=\omega\bigl(8U(k)\bigr)=\omega(U(k))+1.
\end{equation}
Note that both the polynomials \(U\) and \(V\) are nonconstant and primitive.  We already have the following description of the P\'olya group \(\operatorname{Po}(K_t)\) of the field \(K_t\) as
\begin{equation}\label{eq:polya-formula}
\operatorname{Po}(K_t)\cong
\begin{cases}
(\mathbb{Z}/7\mathbb{Z})^{\,\omega(E(t))-2}, & \text{ if $t$ is even},\\[4pt]
(\mathbb{Z}/7\mathbb{Z})^{\,\omega(E(t))-1}, & \text{ if $t$ is odd}.
\end{cases}
\end{equation}
Consequently, combining \eqref{eq:omega-even} with \eqref{eq:polya-formula} yields
\begin{equation}\label{eq:ranks}
\operatorname{rk}_7(\operatorname{Po}(K_{2k})) = \omega(U(k)) - 1,
\qquad
\operatorname{rk}_7(\operatorname{Po}(K_{2k+1})) = \omega(V(k)) - 1.
\end{equation}
Here, for any prime $p$, the $p$-rank of a finite abelian group $G$, denoted $\operatorname{rk}_p(G)$, is the largest integer $\ell$ such that $G$ contains a subgroup isomorphic to $(\mathbb{Z}/p\mathbb{Z})^\ell$. It is enough to construct infinitely many blocks starting at an even index (the case of odd starts is identical). So write the \(m\) consecutive integers as
\[
t=2k,\ 2k+1,\ 2k+2,\ \dots,\ 2k+(m-1)
\]
for some \(k\in\mathbb Z\). Define \(m\) polynomials in variable \(k\) by alternating \(U\) and \(V\) with the following shifts
\[
\begin{aligned}
H_0(k) &:= U(k), \\
H_1(k) &:= V(k), \\
H_2(k) &:= U(k+1), \\
H_3(k) &:= V(k+1), \\
&\ \ \vdots \\
H_{m-1}(k) &:=
\begin{cases}
U\!\bigl(k+\tfrac{m-1}{2}\bigr), \text{ if }\:\: m-1  \text{ is even},\\
V\!\bigl(k+\tfrac{m-2}{2}\bigr), \text{ if } \:\: m-1 \text{ is odd}.
\end{cases}
\end{aligned}
\]
Thus \(H_j(k)\) now corresponds exactly to the value \(E(2k+j)\), i.e., 
if \(j\) is even, \(H_j(k)=U(k+\tfrac{j}{2})\), and if \(j\) is odd,
\(H_j(k)=V(k+\tfrac{j-1}{2})\).
Note that each \(H_j\) is nonconstant and primitive. Now, in view of \eqref{eq:ranks}, we claim that
\begin{equation}\label{eq:target-omega}
\omega\bigl(H_j(k)\bigr)\ \ge\ r+1
\qquad\text{for all }j=0,1,\dots,m-1,
\end{equation}
for infinitely many \(k\). Note that if \eqref{eq:target-omega} holds, then
\(\operatorname{rank}_7\operatorname{Po}\bigl(K_{2k+j}\bigr)\ge r\) for each \(j\), completing the proof.

We recall Schur's lemma \cite[Lemma 3]{Murty2012}, which states that if $H(x) \in \mathbb{Z}[x]$ is nonconstant, then $H$ has infinitely many distinct prime divisors. We say that a prime \(p\) is a {prime divisor} of a polynomial \(H\in\mathbb{Z}[x]\) 
if there exists some \(n\in\mathbb{N}\) such that \(p \mid H(n)\)(see \cite{Murty2012}). By applying this result to each nonconstant polynomial $H_j$, we obtain that there exist infinitely many primes $p$ such that the congruence $H_j(x)\equiv 0 \pmod p$ has a solution. For each \(j\in\{0,\dots,m-1\}\),
we choose a set \(\mathcal P_j\) of \(r+1\) distinct primes with this property, and make these
sets pairwise disjoint.
\[
\mathcal P_j\cap\mathcal P_{j'}=\varnothing\qquad\text{whenever }j\ne j'.
\]
For each \(p\in\mathcal P_j\) choose a residue class \(\alpha_{p,j}\pmod p\) satisfying
\[
H_j(\alpha_{p,j})\equiv 0\pmod p.
\]
Now consider the system of congruences 
\begin{equation}\label{eq:CRT-system}
k\equiv \alpha_{p,j}\pmod p\qquad\text{for every }p\in\bigcup_{j=0}^{m-1}\mathcal P_j.
\end{equation}
All moduli are distinct primes, hence pairwise coprime. By the Chinese remainder theorem, there exists a solution to \eqref{eq:CRT-system}; moreover, the full set of solutions is the infinite arithmetic progression.
\[
k\equiv k_0 \pmod C,\qquad C:=\prod_{j=0}^{m-1}\ \prod_{p\in\mathcal P_j} p.
\]
For any \(k\) satisfying \eqref{eq:CRT-system} we have, for each fixed \(j\),
\[
p\mid H_j(k)\qquad\text{for all }p\in\mathcal P_j,
\]
so the set of distinct prime divisors of \(H_j(k)\) contains \(\mathcal P_j\). Hence
\[
\omega\bigl(H_j(k)\bigr)\ \ge\ |\mathcal P_j|\ =\ r+1
\qquad\text{for every }j=0,1,\dots,m-1,
\]
which is exactly \eqref{eq:target-omega}. This establishes the claim. Because the congruence class \(k\equiv k_0\pmod C\)
contains infinitely many integers \(k\), we obtain infinitely many blocks of
\(m\) consecutive indices \(2k,2k+1,\dots,2k+m-1\) for which all the P\'olya groups have
7--rank at least \(r\). This completes the proof.
\end{proof}

\section{P\'olya Numbers and Monogeneity }

The classical embedding problem in algebraic number theory asks whether every number field $K$ can be embedded in a field $L$ having class number $h_L = 1$. This fundamental question was resolved negatively in 1964 by Golod and Shafarevich \cite{ESG}, who constructed explicit counterexamples. In contrast, the analogous embedding problem for P\'{o}lya fields admits a positive answer. Leriche \cite{AL3} established that every number field can indeed be embedded in a P\'{o}lya field. The key insight underlying this result is Leriche's theorem demonstrating that the Hilbert class field $H_K$ of any number field $K$ is necessarily a P\'{o}lya field (\cite[Theorem 3.3]{AL3}). The P\'olya number of $K$ is given by the integer 
$$po_K=\text{min}\{[L : K]\mid K\subseteq L,\: L\text{ P\'olya field}\}.$$
The following theorem establishes an upper bound for the P\'olya numbers of cyclic number fields in terms of the size of their P\'olya groups. 
\begin{prop}\label{c1}\cite[Corollary 2.5.1]{thesis}
    Let $K$ be a cyclic number field of odd degree, then $po_{K}\leq |\text{Po}(K)|$.
\end{prop}
As an immediate corollary, we obtain the following result.
\begin{cor}
 Let $K_t$ denote the Hashimoto--Hoshi cyclic septic field. Then $po_{K_t}\leq|\text{Po}(K_t)|$.   
\end{cor}
Let $K$ be a number field and $\theta \in \mathcal{O}_K$ a primitive element of $K$. The index of $\theta$ in $K$ is defined as $I(\theta) = [\mathcal{O}_K : \mathbb{Z}[\theta]]$ and the index of $K$ is defined as 
\[
I(K) = \gcd\{I(\theta) : \theta \in \mathcal{O}_K \text{ with } K = \mathbb{Q}(\theta)\}.
\]
When $I(K) > 1$, the field $K$ cannot be monogenic, that is, $\mathcal{O}_K \neq \mathbb{Z}[\theta]$ for every $\theta \in K$. However, the converse does not hold in general. For instance, this occurs for the Lehmer quintic fields (see \cite{NKM}). In this direction, we prove the following theorem.
\begin{thm}\label{N100}
Let $t$ be an odd integer coprime to 15, and let $K_t$ denote the corresponding Hashimoto--Hoshi cyclic septic field. Then $K_t$ is non-monogenic with index $I(K_t)=1$.
\end{thm}
\begin{proof}
It is known that $\mathrm{Gal}(K_t/\mathbb{Q}) \simeq \mathbb{Z}/7\mathbb{Z}$. Using \thmref{GRAS}, a cyclic number field of prime degree $\ell \geq 5$ can be monogenic only if $2\ell+1$ is prime and the field is the maximal real subfield of $\mathbb{Q}(\zeta_{2\ell+1})$. For $\ell=7$, this would require $2\ell+1=15$ to be prime, which is absurd. Therefore, $K_t$ is not monogenic.
We now prove that $I(K_t)=1$ whenever $t$ is odd and coprime to $15$.
Recall the fundamental relation.
\[
\operatorname{disc}(f_t) \;=\; [I(\theta_t)]^2 \, d(K_t).
\]
Using \textsc{Maple}, we compute explicitly
\begin{align*}
\operatorname{disc}(f_t) 
&= t^{22}\,
\bigl(t^{6}+2t^{5}+11t^{4}+t^{3}+16t^{2}+4t+8\bigr)^{6} \\
&\quad \times \bigl(2t^{4}-2t^{3}+6t^{2}-3t+4\bigr)^{2}\,
\bigl(t^{5}+t^{4}+t^{3}+2t^{2}+t+1\bigr)^{2} \\
&= t^{22}\,(E(t))^6\,
\bigl(2t^{4}-2t^{3}+6t^{2}-3t+4\bigr)^{2}\,
\bigl(t^{5}+t^{4}+t^{3}+2t^{2}+t+1\bigr)^{2}.
\end{align*}
Now suppose $t$ is odd and $\gcd(t,15)=1$.  
A direct check modulo $2$, $3$, and $5$ shows that none of the three polynomial factors
\[
E(t), \quad 2t^{4}-2t^{3}+6t^{2}-3t+4, \quad 
t^{5}+t^{4}+t^{3}+2t^{2}+t+1
\]
vanishes modulo $2$, $3$, or $5$. Hence, $\operatorname{disc}(f_t)$ is not divisible by $2$, $3$, or $5$. Since the index $I(K_t)$ can only have prime divisors strictly less than the degree $[K_t:\mathbb{Q}]=7$ (see \cite{ZAL}), we conclude that $I(K_t)=1$.
\end{proof}

\section{Concluding Remarks}

In this article, we have provided a complete characterization of the P\'olya property for the 
Hashimoto--Hoshi family of cyclic septic fields. We showed that the structure of their P\'olya groups is governed entirely by the prime factorization of the values of the polynomial
\[
E(t) = t^6 + 2t^5 + 11t^4 + t^3 + 16t^2 + 4t + 8
\]
at integers $t$, and that there are infinitely many non-P\'olya fields whose P\'olya groups can be made arbitrarily large. Furthermore, we established the existence of arbitrarily long blocks of consecutive fields whose P\'olya groups are simultaneously large.  However at present, in the view of \rmkref{r3}, the infinitude of P\'olya fields in this family remains conditional.

To analyze this further, it is natural to study the Diophantine equations arising from perfect powers of $E(t)$. This leads us to consider the superelliptic curves
\[
C_n : Y^n = E(X), \qquad n=2,3,4.
\]
Since $E(X)$ has degree $6$, an application of the Riemann--Hurwitz formula (see \cite[Theorem~4.16]{Miranda1995}) shows that these curves have genus $2$, $4$, and $7$ for $n=2,3$, and $4$, respectively. As these genera all exceed $1$, Faltings' theorem \cite[Theorem~E.0.1]{FALT} implies that each of the above equations admits only finitely many integral solutions. In particular, there exist only finitely many integers $t$ such that
\[
E(t) = p^2, \quad E(t) = p^3, \quad \text{or} \quad E(t) = p^4,
\]
for some prime $p$. Consequently, under the assumption that $E(t)$ is fifth--power free, the problem of determining whether this family produces infinitely many P\'olya fields is directly linked to a deep open problem concerning prime values of polynomials. This leads naturally to the following question, which is of considerable significance. 

\noindent\textbf{Question.} 
Is it possible to prove unconditionally that infinitely many Hashimoto--Hoshi fields under the assumption that $E(t)$ fifth--power free are P\'olya?

A positive answer to this question would immediately settle Bunyakovsky’s conjecture for the polynomial $E(t)$. It is worth noting that, to date, this conjecture is not known for any single polynomial of degree greater than one.

\appendix
\section*{Appendix: \textsc{Maple} Computations}

\subsection*{A. Tschirnhaus Transformation and Discriminant}
The \textsc{Maple} script below computes the Tschirnhaus-transformed polynomial $f^*_t$, displays its coefficients in factorized form as in~\eqref{common}, and evaluates the discriminant of $g_t$ given in~\eqref{disc}.
\begin{verbatim}
a6 := -t^3 - t^2 - 5*t - 6:
a5 := 3*(3*t^3 + 3*t^2 + 8*t + 4):
a4 := t^7 + t^6 + 9*t^5 - 5*t^4 - 15*t^3 - 22*t^2 - 36*t - 8:
a3 := -t*(t^7 + 5*t^6 + 12*t^5 + 24*t^4 - 6*t^3 + 2*t^2 - 20*t - 16):
a2 := t^2*(2*t^6 + 7*t^5 + 19*t^4 + 14*t^3 + 2*t^2 + 8*t - 8):
a1 := -t^4*(t^4 + 4*t^3 + 8*t^2 + 4):
a0 := t^7:
ft := expand(X^7 + a6*X^6 + a5*X^5 + a4*X^4 + a3*X^3 + a2*X^2 
             + a1*X + a0):
c := t^3 + t^2 + 5*t + 6:
fstar_raw := subs(X = (X + c)/7, ft):
fstar := expand(denom(fstar_raw) * fstar_raw);
h5 := coeff(fstar, X, 5):
h4 := coeff(fstar, X, 4):
h3 := coeff(fstar, X, 3):
h2 := coeff(fstar, X, 2):
h1 := coeff(fstar, X, 1):
h0 := coeff(fstar, X, 0):
for k from 5 to 0 by -1 do
    hk := parse(cat("h", k));
    printf("h%d = %a\n", k, hk);
    printf("factor(h%d) = %a\n\n", k, factor(hk));
end do:
gt := expand(subs(X = m*X,
                  X^7 + h5*X^5 + h4*X^4 + h3*X^3 + h2*X^2 + h1*X 
                  + h0)/m^5):
disc_gt := discrim(gt, X):
disc_factored := factor(disc_gt):
printf("g_t(X) = %a\n", gt);
printf("\n");
printf("Discriminant of g_t(X) (factored) = %a\n", disc_factored);
\end{verbatim}
\subsection*{B. Verification of the Relation in \texorpdfstring{\eqref{eq:bezout}}{(eq. bezout)}}
The following \textsc{Maple} script checks the explicit relation between $E$ and $H$ that appears in \eqref{eq:bezout}.
\begin{verbatim}
E := t^6 + 2*t^5 + 11*t^4 + t^3 + 16*t^2 + 4*t + 8:
H := 6*t^15 + 30*t^14 - 133*t^13 - 504*t^12
     - 3255*t^11 - 6244*t^10 - 12033*t^9 - 8438*t^8
     + 19620*t^7 + 52892*t^6 + 136787*t^5
     + 167671*t^4 + 179676*t^3 + 206640*t^2
     + 103680*t + 82944:
L := -150*t^14 - 492*t^13 + 5035*t^12 + 7709*t^11 + 43169*t^10
     + 6075*t^9 - 61261*t^8 - 3409*t^7 - 53980*t^6 + 1520256*t^5
     + 1587001*t^4 - 1436918*t^3 - 5167741*t^2 - 16768814*t 
     - 14725412:
M := 25*t^5 + 7*t^4 + 119*t^3 - 376*t^2 + 155*t + 1738:
g := gcd(E, H):
printf("gcd(E,H) = %a\n", g);
out := expand(E*L + H*M):
printf("E*L + H*M = %a\n", out);
printf("Prime factorization = %a\n", ifactor(out));
\end{verbatim}
\subsection*{C. 7-adic Valuations}
The next \textsc{Maple} script computes the $7$-adic valuations of the polynomials in~\eqref{value} for the residue class $t \equiv 2 \pmod{7}$.
\begin{verbatim}
P_E := t -> t^6 + 2*t^5 + 11*t^4 + t^3 + 16*t^2 + 4*t + 8:
P_F := t -> 10*t^3 + 10*t^2 + t + 4:
P_G := t -> 15*t^6 + 30*t^5 - 31*t^4 + 15*t^3 - 201*t^2 - 87*t - 174:
P_H := t -> 6*t^15 + 30*t^14 - 133*t^13 - 504*t^12 - 3255*t^11
            - 6244*t^10 - 12033*t^9 - 8438*t^8 + 19620*t^7 
            + 52892*t^6 + 136787*t^5 + 167671*t^4 + 179676*t^3
            + 206640*t^2+ 103680*t + 82944:
P_I := t -> 12*t^9 + 36*t^8 - 78*t^7 - 84*t^6 - 861*t^5 - 588*t^4
            - 1155*t^3 - 1214*t^2 - 324*t - 432:
P_J := t -> 5*t^12 + 20*t^11 - 66*t^10 - 162*t^9 - 1126*t^8
            - 1441*t^7 - 2534*t^6 - 1641*t^5 + 1857*t^4 
            + 426*t^3 + 5574*t^2 + 3456*t + 3456:
polys := [P_E, P_F, P_G, P_H, P_I, P_J]:
names := ["E", "F", "G", "H", "I", "J"]:
resid := 2:
for idx to nops(polys) do
    P := polys[idx]:
    pname := names[idx]:
    P_at_res := eval(P(resid)):
    if P_at_res = 0 then
        printf("%s: P(%d) = 0 => 7-adic valuation must be computed
        from coefficients (vanishes at t=%d).\n", pname, resid,
        resid):
        next;
    end if;
    v_const := 0:
    tmp := P_at_res:
    while tmp mod 7 = 0 do
        tmp := tmp/7:
        v_const := v_const + 1:
    end do:
    if v_const = 0 then
        printf("%s: v7(P(%d)) = 0. Not divisible by 7 for
        t &equiv; %d (mod 7).\n", pname, resid, resid):
        next;
    end if;
    found_val := 0:
    for k to v_const do
        modulus := 7^k:
        all_zero := true:
        for m from 0 to 7^(k - 1) - 1 do
            r := resid + 7*m:
            if eval(P(r)) mod modulus <> 0 then
                all_zero := false:
                break:
            end if:
        end do;
        if all_zero then
            found_val := k:
        else
            break:
        end if:
    end do:
    printf("The 7-adic valuation of %s for t congruent to 2 modulo
    7 is %d\n", pname, found_val):
end do:
\end{verbatim}

\subsection*{D. Verification of the Polynomial Identity in (\ref{neww})}
The following \textsc{Maple} script can be used to verify the polynomial identity in (\ref{neww}). Evaluating the last line in \textsc{Maple} returns `0', confirming that the identity is satisfied.
\begin{verbatim}
Phi7 := x -> x^6 + x^5 + x^4 + x^3 + x^2 + x + 1;
w := x -> (2*x^5 + 3*x^4 + 18*x^3 - 11*x^2 + 13*x)/14;
Ex := x -> x^6 + 2*x^5 + 11*x^4 + x^3 + 16*x^2 + 4*x + 8;
W := x -> 64*x^24 + 448*x^23 + 4016*x^22 + 15104*x^21 + 71180*x^20 
          + 142780*x^19 + 480245*x^18 + 258262*x^17 + 1611259*x^16 
          - 1664243*x^15 + 6281164*x^14 - 9116196*x^13 
          + 20556660*x^12 - 28970654*x^11 + 33948931*x^10
          - 21216648*x^9 + 5083933*x^8 + 7056657*x^7 - 3702019*x^6
          - 2324868*x^5 + 4160443*x^4 + 671937*x^3 - 2012038*x^2 
          + 403368*x + 941192;
simplify(Phi7(w(x)) - (Ex(x) * W(x))/(2^6 * 7^6));
\end{verbatim}

\section*{Acknowledgements}
The author expresses sincere gratitude to Dr. Srilakshmi Krishnamoorthy for her constant encouragement and guidance during this project. Special thanks are also extended to Dr. Mahesh Kumar Ram for his valuable suggestions. The author also acknowledges the Indian Institute of Science Education and Research Thiruvananthapuram for providing financial support and an excellent research environment to carry out this work.

\section*{Data availability}
Data sharing is not applicable to this article as no datasets were generated or analysed during the current study.

\end{document}